\documentclass[12pt,centertags]{amsart}
\usepackage{amsmath,amstext,amsthm,a4,amssymb,amscd}
\usepackage[mathscr]{eucal}
\usepackage{mathrsfs}
\usepackage{epsf}
\numberwithin{equation}{section}

 \DeclareMathOperator{\Dom}{Dom}
 
 \DeclareMathOperator{\supp}{supp}

\newtheorem{thm}{Theorem}[section]
\newtheorem{lemma}[thm]{Lemma}
\newtheorem{prop}[thm]{Proposition}
\newtheorem{cor}[thm]{Corollary}
\theoremstyle{definition}
\newtheorem{rem}[thm]{Remark}
\theoremstyle{definition}

\newcommand{\be}{\begin{eqnarray}}
\newcommand{\ee}{\end{eqnarray}}

\newcommand{\comment}[1]{}

\begin{document}

\title{A Thom-Smale-Witten theorem\\
on manifolds with boundary}
\author{Wen  Lu}

\address{Mathematisches Institut,
Universit\"at zu K\"oln, Weyertal 86-90, 50931 K\"oln, Germany}
\email{wlu@math.uni-koeln.de}

\date{\today}

\maketitle

\begin{abstract}
Given a smooth compact manifold with boundary,
we show that the subcomplex
of the deformed de Rham complex consisting of eigenspaces of small eigenvalues of the Witten
Laplacian is canonically isomorphic to the Thom-Smale complex constructed by Laudenbach
in \cite{Laudenbach10}. Our proof is based on
Bismut-Lebeau's analytic localization techniques.
As a by-product, we obtain Morse inequalities for manifolds
with boundary.
\end{abstract}

\section{Introduction}
\label{s0}

\noindent Morse theory is a method to determine the topology of a finite or infinite dimensional manifold
(such as the space of paths or loops on a compact manifold) from the critical points of
one suitable function on the manifold. The theory has many far-reaching applications
ranging from existence of exotic spheres to supersymmetry and Yang-Mills theory.

In this paper we will study Morse theory on a manifold with boundary.
The main goal is to exhibit a canonical isomorphism
between the Witten instanton complex constrained to boundary conditions
and the Thom-Smale complex.

Recall that a Morse function on a manifold without boundary is a
smooth real function whose critical points are all non-degenerate.
Let $M$ be a smooth $n$-dimensional closed manifold.
Let $f$ be a Morse function on $M$ and choose a Riemannian metric on $M$ such that
the gradient vector field $\nabla f$ satisfies the Morse-Smale transversality conditions.
This dynamical system gives rise to a chain complex, called Thom-Smale complex,
having the critical points as generators and boundary map expressed in terms of the unstable
and stable manifolds.
The Morse homology theorem asserts that the Thom-Smale complex recovers the standard homology of
the underlying manifold (see \cite{Laudenbach92,schwarz93} and the colourful historical presentation \cite{Bott88}).

In \cite{Witten82} Witten suggested that
the Thom-Smale complex could be recovered from the Witten instanton complex,
which is subcomplex of the deformed de Rham complex, consisting of eigenspaces of small
eigenvalues of the Witten Laplacian $D^{2}_{T}$
(cf. (\ref{1.4a})). This fact was first established rigorously by Helffer and Sj\"ostrand \cite[Prop.\,3.3]{Hellfer85}.

Later, Bismut and Zhang \cite{Bismut92} generalized the results
of Helffer and Sj\"ostrand for the Witten instanton complex
with values in a flat vector bundle.  This is an important step in their proof
and extension of the Cheeger-M\"uller theorem \cite{Cheeger79,M78} about
the equality of the Reidemeister and Ray-Singer metrics.

In \cite[\S 6]{Bismut94}, Bismut and Zhang gave a new and simple proof of
the isomorphism between the Witten instanton complex and the Thom-Smale complex by means of
resolvent estimates. We refer the readers to \cite[Chapter 6]{Zhang01} for a comprehensive study of the Witten
deformation following Bismut-Zhang's approach.

Let us go back to the general case of a compact (not necessarily orientable) manifold $M$
of dimension $n$ with boundary $\partial M\neq \emptyset$.
In this paper, a smooth function
$f: M\rightarrow \mathbb{R}$ is called a Morse function if the restrictions of $f$ to the interior and
boundary of $M$ are Morse functions in the usual sense and if $f$ has no critical point on $\partial M$
(cf.\ \cite{Chang95,Laudenbach10}).

In order to have Hodge theory for the de Rham Laplacian, we will
impose boundary conditions: absolute
boundary condition and relative boundary condition (cf.\ \cite[pp.\,361-371]{Taylor96}).
Denote by $\mathcal{H}_{a}$ (resp.\ $\mathcal{H}_{r}$)
the subspace of harmonic forms satisfying
the absolute boundary condition (resp.\ the relative boundary condition). Then
the space $\mathcal{H}_{a}$ (resp.\ $\mathcal{H}_{r}$) is isomorphic to the absolute
cohomology group $H_{dR}^{\bullet}(M, \mathbb{R})$
(resp. the relative cohomology group $H_{dR}^{\bullet}(M, \partial M; o(TM))$ with coefficients
twisted by the orientation bundle $o(TM)$ of $M$).

In \cite{Chang95}, Chang and Liu established the Morse inequalities
corresponding to the absolute and relative boundary conditions for orientable manifolds,
by using Witten deformation.

On the other hand, Laudenbach \cite{Laudenbach10} recently constructed a Thom-Smale complex
whose homology is isomorphic to the (absolute or relative) homology of $M$ with integral coefficients.
The construction uses a pseudo-gradient Morse-Smale vector field suitably adapted to the boundary.

We will prove that the Witten instanton complex constrained to the boundary conditions
and the Thom-Smale complex constructed by Laudenbach in \cite{Laudenbach10} are
canonically isomorphic.
This result is new for manifolds with non-empty boundary, while its
counterpart for closed manifolds appeared in \cite{Bismut92,Hellfer85}.
Our method consists of applying Bismut-Lebeau's localization techniques
\cite{Bismut74} along the lines of \cite{Zhang01}.

In order to state the results let us introduce some notations.
Let $f: M\rightarrow \mathbb{R}$ be a Morse function on $M$, and let $f|_{\partial M}$ be
its restriction to the boundary.
Let $C^{j}(f)$ (resp.\ $C^{j}(f|_{\partial M})$ be the set consisting of all critical points
of $f$ (resp.\ $f|_{\partial M}$) with index $j$.
Let $\nu$ be the outward normal vector field along $\partial M$. Denote by
\begin{equation*}
C_{-}^{j}(f|_{\partial M})=\{p\in C^{j}(f|_{\partial M}): (\nu f)(p)<0\}\,,\;\;
C_{+}^{j}(f|_{\partial M})=\{p\in C^{j}(f|_{\partial M}):  (\nu f)(p)>0\}\,,
\end{equation*}
and
\[
c_{j}=\#\,C^{j}(f)\,,\;\;p_{j}=\#\,C_{-}^{j}(f|_{\partial M})\,,\;\;q_{j}=\#\,C_{+}^{j}(f|_{\partial M})\,.
\]
Set $C(f)=\bigsqcup_{j=1}^{n}C^{j}(f)$ and
$C_{-}(f|_{\partial M})=\bigsqcup_{j=0}^{n-1}C_{-}^{j}(f|_{\partial M})$.

In order to define the Witten instanton complex we have to state two results about the spectrum of the Witten Laplacian. Let
\begin{align}
0\leqslant \lambda_{1}^j(T)\leqslant \lambda_{2}^j(T)\leqslant\cdots
\end{align}
be the eigenvalues of the Witten Laplacian $D_T^2$
with absolute boundary condition acting on $j$-forms.

\begin{thm}\label{t1.3}
There exists positive constants $a_{1}$ and $T_{1}$ such that for $T>T_{1}$,
and $j=0, 1, \cdots, n$, we have
\begin{align}\label{0.1}
\lambda_{\ell}^j(T)\geqslant a_{1}T^{2},\ \ \ \textup{for}\ \ell\geqslant c_{j}+p_{j}+1.
\end{align}
\end{thm}
The estimate (\ref{0.1}) was obtained by Chang and Liu (cf. \cite[\S 3, Th.\,2]{Chang95})
by localization and the min-max principle 
and by Le Peutrec (cf. \cite[Th.\,3.5]{Peutrec10})
by delicate constructions of quasimodes and the WKB method.
We will give here a short proof based on elementary spectral estimates.

Our second result provides a refined estimate of the lower part of the spectrum
of the Witten Laplacian.
\begin{thm}\label{t1.4}
There exist positive constants $a_{2}, a_{3}$ and $T_{2}$ such that for $T\geqslant T_{2}$ and
$j=0, 1, \cdots, n$,
\begin{align}\label{0.2}
\lambda_{\ell}^j(T)\leqslant a_{2}e^{-a_{3}T}, \ \ \textup{for}\ \ell\leqslant c_{j}+p_{j}.
\end{align}
\end{thm}
The estimate (\ref{0.2}) for $j=0$ was obtained by D. Le Peutrec via
WKB analysis (cf.\ \cite[Th.\,1.0.3]{Peutrec10}).

From Theorems \ref{t1.3} and \ref{t1.4} follows that the Witten Laplacian $D_T^2$
acting on $j$-forms has a spectral gap: the upper part of the spectrum grows with speed
$T^2$ and the lower part decays exponentially in $T$. Moreover,
for a given $C_{0}>0$, there exists $T_{0}>0$ such that for $T\geqslant T_{0}$,
the number of eigenvalues in $[0, C_{0})$ equals
$c_{j}+p_{j}$ (see Proposition \ref{t1.2}).
Let $F^{C_{0}}_{T, j}$ denote the
$(c_{j}+p_{j})$-dimensional vector space generated by the eigenspaces associated to the
eigenvalues lying in [0, $C_{0}$). It is easy to see that $(F^{C_{0}}_{T,\bullet}, d_{T})$ forms a complex, called
Witten instanton complex.

Let $(C^{\bullet}, \partial)$ denote the Thom-Smale complex constructed
in \cite{Laudenbach10} (cf.\ (\ref{2.32c})--(\ref{2.32d})).

Let $P_{\infty}$ be the natural morphism from the de Rham complex $(\Omega^{\bullet}(M), d)$ to
the Thom-Smale complex $(C^{\bullet}, \partial)$, defined by integration on the closure of the unstable manifolds
(cf. \S\ref{s5}):
\begin{align}\label{7.7b}
P_{\infty}(\alpha)=\sum_{p\in C(f)\cup C_{-}(f|_{\partial M})}[p]^{\ast}\int\limits_{\overline{W^{u}(p)}}\alpha
\in C^{\bullet},\ \ \text{for $\alpha\in \Omega^{\bullet}(M)$}\,.
\end{align}
Set
\begin{align}\label{7.7a}
P_{\infty, T}(\alpha)=P_{\infty}(e^{Tf}\alpha), \ \textup{for}\ \alpha\in F^{C_{0}}_{T,\bullet}\,.
\end{align}
The main result of this paper is as follows.
\begin{thm}\label{t1.0}
The map $P_{\infty, T}: (F^{C_{0}}_{T,\bullet}, d_{T})
\rightarrow (C^{\bullet}, \partial)$ is an isomorphism of complexes for $T$ large enough.
In particular, the chain map $P_{\infty}$
is a quasi-isomorphism between the de Rham complex $(\Omega^{\bullet}(M),d)$
and the Thom-Smale complex $(C^{\bullet},\partial)$.
\end{thm}
Due to the de Rham isomorphism $H^{\bullet}_{dR}(M,\mathbb{R})\cong H^{\bullet}(M,\mathbb{R})$
between de Rham cohomology and singular cohomology, we have thus an analytic proof of
the fact that the cohomology of the Thom-Smale complex $(C^{\bullet},\partial)$ coincides with the singular cohomology of $M$.
This was proved by topological methods by Laudenbach \cite{Laudenbach10}.

Let $\beta_{j}(M)$ denote the $j$-th Betti number of the de Rham complex, i.e.,
$\beta_{j}(M)=\dim H^{j}_{dR}(M, \mathbb{R})$.
As a by-product of the proof of Theorem \ref{t1.0}, we obtain the following Morse inequalities
for manifolds with boundary:

\begin{cor}\label{t1.1}
For any $k=0,1,\cdots, n$, we  have
\begin{align}\label{1.1}
\sum_{j=0}^{k}(-1)^{k-j}\beta_{j}(M) &\leqslant \sum_{j=0}^{k}(-1)^{k-j}(c_{j}+p_{j}),
\end{align}
with equality for $k=n$.
\end{cor}

The Morse inequalities (\ref{1.1}) were first established in \cite[\S 4]{Chang95}.
A topological proof of (\ref{1.1}) appeared in \cite{Laudenbach10}. Denote by $\beta_{j}(M, \partial M)$
the $j$-th Betti number of the relative de Rham complex with coefficients twisted by the orientation bundle,
i.e., $\beta_{j}(M, \partial M)=\dim H^{j}_{dR}(M, \partial M; o(TM))$.
By replacing the Morse function $f$ by $-f$
and applying the Poincar\'e duality theorem, we get
from Corollary \ref{t1.1} the following result.
\begin{cor} \label{t1.2a}
For any $k=0,1,\cdots, n$, we  have
\begin{align}\label{1.1a}
\sum_{j=0}^{k}(-1)^{k-j}\beta_{j}(M, \partial M) &\leqslant \sum_{j=0}^{k}(-1)^{k-j}(c_{j}+q_{j-1}),
\end{align}
with equality for $k=n$.
\end{cor}

On the other hand, if we endow the Witten Laplacian $D^{2}_{T}$ with relative boundary
condition, we will obtain analogue results of Theorems \ref{t1.3}, \ref{t1.4} and \ref{t1.0}.
Then we also derive the inequalities (\ref{1.1a}). See Remark \ref{t7.2}.

This paper is organized as follows. In Section \ref{s5}, we introduce
the Thom-Smale complex constructed by Laudenbach in \cite{Laudenbach10}.
Section \ref{s1} is devoted to
some calculations of the kernels of the Witten Laplacian on Euclidean spaces.
The results will be applied to the Witten Laplacian
on manifolds around the critical points in $C(f)$ and $C_{-}(f|_{\partial M})$.
In Section \ref{s2}, we state a crucial result (Proposition \ref{t1.2}) concerning the lower part of
the spectrum of the Witten Laplacian for $T$ large.
We also prove Corollaries \ref{t1.1} and \ref{t1.2a} there.
In Section \ref{s6}, we prove Proposition
\ref{t1.2} and then finish the proof of Theorem \ref{t1.0}.

\section{The Thom-Smale complex constructed by Laudenbach}\label{s5}

In this section, we introduce the Thom-Smale complex constructed by Laudenbach in \cite{Laudenbach10}.
By \cite[\S 2.1]{Laudenbach10}, there exists a vector field $X$ on $M$ satisfying the following conditions.
\\ (1) $(Xf)(\cdot)<0$ except at critical points in $C(f)\cup C_{-}(f|_{\partial M})$;
\\ (2) $X$ points inwards along $\partial M$ except in a neighborhood in $\partial M$ of critical points
in $C_{-}(f|_{\partial M})$ where it is tangent to $\partial M$;
\\ (3) if $p\in C^{j}(f)$, then there exists a coordinate system $(x, U_{p})$ such that on $U_{p}$ we have
\begin{align}\label{2.3a}
f(x)=f(p)-\frac{x^{2}_{1}}{2}-\cdots-\frac{x^{2}_{j}}{2}+\frac{x^{2}_{j+1}}{2}+\cdots+\frac{x^{2}_{n}}{2}\,,
\end{align}
and
\begin{align}
X=\sum_{i=1}^{j}x_{i}\frac{\partial}{\partial x_{i}}-\sum_{i=j+1}^{n}x_{i}\frac{\partial}{\partial x_{i}}\,;
\end{align}
\\ (4) if $p\in C_{-}^{j}(f|_{\partial M})$, then there are coordinates
$x=(x', x_{n})\in \mathbb{R}^{n-1}\times \mathbb{R}_{+}$
on some neighborhood $U_{p}$ of $p$ such that on $U_{p}$,
\begin{align}\label{2.3b}
f(x)=f(p)-\frac{x^{2}_{1}}{2}-\cdots-\frac{x^{2}_{j}}{2}+\frac{x^{2}_{j+1}}{2}+\cdots+\frac{x^{2}_{n-1}}{2}+x_{n}
\end{align}
and
\begin{align}
X=\sum_{i=1}^{j}x_{i}\frac{\partial}{\partial x_{i}}-\sum_{i=j+1}^{n}x_{i}\frac{\partial}{\partial x_{i}};
\end{align}
\\ (5) $X$ is Morse-Smale in the sense that the global unstable manifolds and the local stable manifolds
are mutually transverse. Denote by $W^{u}(p)$ (resp.\ $W_{loc}^{s}(p)$) the unstable manifold
(resp.\ the local stable manifold) of $p$ which by definition, consists of all the
flow lines of $X$ that emanate from $p$
(resp. end at $p$).
We denote by $\overline{W^{u}(p)}$ the closure of $W^{u}(p)$ in $M$.

If $p\in C^{j}(f)$, then we get from the condition (3) that locally,
\begin{align}\label{2.32a}
W^{u}(p)=\Big\{(x_{1}, \cdots, x_{j}, 0, \cdots, 0)\Big\}\subset \mathbb{R}^{n}.
\end{align}
\noindent In view of the conditions (1) and (2),
$W^{u}(p)$ does not intersect with the boundary except at critical points in $C_{-}(f|_{\partial M})$.

If $p\in C_{-}^{j}(f|_{\partial M})$, then the condition (4) implies that
\begin{align}\label{2.32b}
W^{u}(p)=\Big\{(x_{1}, \cdots, x_{j}, 0, \cdots, 0)\Big\}\subset \mathbb{R}^{n-1}\times \mathbb{R}_{+}\,,
\end{align}
that is,
$W^{u}(p)$ lies completely in the closed submanifold $\partial M$.
Therefore, the results \cite[Prop.\,2]{Laudenbach92}
about the structure of $\overline{W^{u}(p)}$ still hold for $p\in C^{j}(f)\cup C_{-}^{j}(f|_{\partial M})$, i.e.,
$\overline{W^{u}(p)}$ is a $j$-dimensional submanifold of $M$ with conical singularities and
$\overline{W^{u}(p)}\backslash W^{u}(p)$ is stratified by unstable manifolds of critical points of
index strictly less than $j$.

If $q\in C^{j+1}(f)$ and $p\in C^{j}(f)$ (resp.\
$q\in C^{j+1}(f)$ and $p\in  C_{-}^{j}(f|_{\partial M})$, resp.\
$q\in  C_{-}^{j+1}(f|_{\partial M})$ and $p\in  C_{-}^{j}(f|_{\partial M})$),
then it is a consequence of the condition (5) that
the intersection of $W^{u}(q)$ with $W_{loc}^{s}(p)$ consists of a finite set of
integral curves of $X$. Choose an orientation on $W^{u}(q)$ and $W^{u}(p)$, respectively.
The orientation of $W_{loc}^{s}(p)$ is uniquely determined. Take $\gamma \in W^{u}(q)\cap W_{loc}^{s}(p)$.
For $x\in \gamma$, denote by $B_{x}$ the sets such that $(X_{x}, B_{x})$ forms an positive
basis of $T_{x}\big(W^{u}(q)\big)$. Let $A_{x}$ denote a positive basis of
$T_{x}\big(W_{loc}^{s}(p)\big)$. Then $(A_{x}, B_{x})$ forms a basis of $T_{x}M$. Set $n_{\gamma}(q, p)=1$
if $(A_{x}, B_{x})$ denotes a positive orientation of $T_{x}M$. Otherwise take $n_{\gamma}(q, p)=-1$.
Let $n(q, p)$ be the sum of $n_{\gamma}(q, p)$ over all integral curves $\gamma$ from $q$ to $p$.
Set $n(q, p)=0$ if $q\in C_{-}^{j+1}(f|_{\partial M})$ and $p\in C^{j}(f)$.

For $p\in C^{j}(f)\cup C_{-}^{j}(f|_{\partial M})$, let $[p]$ be the real line generated by $p$, and let
$[p]^{\ast}$ be the line dual to $[p]$. As in \cite{Laudenbach10}, set
\begin{align}\label{2.32c}
C^{j}=\bigoplus_{p\in C^{j}(f)\cup C_{-}^{j}(f|_{\partial M})}[p]^{\ast}.
\end{align}
The boundary morphism $\partial$ from $C^{j}$ to $C^{j+1}$ is given by
\begin{align}\label{2.32d}
\partial [p]^{\ast}=\sum_{q\in C^{j+1}(f)\cup C_{-}^{j+1}(f|_{\partial M})}n(q, p)[q]^{\ast}.
\end{align}
It is a consequence of \cite[\S 2.2, Prop.\,]{Laudenbach10}
that $(C^{\bullet}, \partial)$ is a chain complex.

For $a>0$, denote by $B^{M}(p, 4a)$ the open ball in $M$ centered at the point
$p\in M$ with the radius $4a$.
In the sequel, we always take for simplicity $U_{p}=B^{M}(p, 4a)$ in the condition (2)
and $U_{p}=B^{M}(p, 4a)\cap \partial M$ in the condition (3).

\section{Some Calculations on Vector Spaces}\label{s1}

In this section, we calculate the kernels of the Witten Laplacian on vector spaces.
The results  will be applied to the Witten Laplacian
around the critical points in $C^{j}(f)$ and $C^{j}_{-}(f|_{\partial M})$.

\subsection{The Witten Laplacian on $\mathbb{R}^{n}$}

Let $V$ be an $n$-dimensional real vector space endowed with an
Euclidean scalar product.  Let $V^{-}, V^{+}$ be two subspaces such
that $V=V^{-} \oplus V^{+}$ and $\textrm{dim} \ V^{-}=j$. Take
$e_{1}, \cdots, e_{n}$ as an orthonormal basis on $V$ such that
$V^{-}$ is spanned by $e_{1}, \cdots, e_{j}$. Let $f$ be a
smooth function on $V$ given by:
\begin{align}\label{2.3}
f(Z)=f(0)-\frac{{\left\vert Z^{-}\right\vert}^2}{2}+
\frac{{\left\vert Z^{+}\right\vert}^2}{2}, \end{align} where
$Z^{-}=(Z_{1}, \cdots, Z_{j}), Z^{+}=(Z_{j+1}, \cdots,
Z_{n})$ and $(Z^{-}, Z^{+})$ denotes the coordinate functions on $V$
corresponding to the decomposition $V=V^{-}\oplus V^{+}$.

Let $Z=\sum_{\alpha=1}^{n}Z_{\alpha}e_{\alpha}$
be the radial vector field on $V$. There is a
natural Euclidean scalar product on $\Lambda V^{\ast}$. Let
$dv_{V}(Z)$ be the volume form on $V$. Denote by $L^{2}(\Lambda V^{\ast})$
the set of the square integrable sections of $\Lambda V^{\ast}$ over $V$. For
$w_{1}, w_{2} \in L^{2}(\Lambda V^{\ast})$, set
\begin{align}\label{2.4}
\big\langle w_{1}, w_{2}\big\rangle=
\int_{V}\langle w_{1}, w_{2}\rangle_{\Lambda V^{\ast}}dv_{V}(Z).
\end{align}

Let $C(V)$ be the Clifford algebra of $V$, i.e., the algebra
generated over $\mathbb R$ by $e\in V$ and the commutation
relations $e\cdot e'+e'\cdot e=-2\langle e, e'\rangle$ for $e, e'\in V$.
Let $c(e), \widehat{c}(e)$ be the Clifford
operators acting on $\Lambda V^{\ast}$ defined by
\begin{align} \label{2.5}c(e)=e^{\ast}\wedge-i_{e}, \
\widehat{c}(e)=e^{\ast}\wedge+i_{e}, \end{align} where
$e^{\ast}\wedge$ and $i_{e}$ are the standard notations for exterior
and interior multiplication and $e^{\ast}$ denotes the dual of $e$
by the Euclidean scalar product on $V$. Then $\Lambda V^{\ast}$ is
a Clifford module.

If $X, Y\in V$, then
\begin{align}\label{2.6}
c(X)c(Y)+c(Y)c(X)&=-2\langle X, Y\rangle, \nonumber\\
\widehat{c}(X)\widehat{c}(Y)+\widehat{c}(Y)\widehat{c}(X)&=2\langle X, Y\rangle,\\
c(X)\widehat{c}(Y)+\widehat{c}(Y)c(X)&=0.  \nonumber \end{align}
Let $d$ be the exterior differential derivative acting on the
smooth sections of $\Lambda V^{\ast}$, and let $\delta$ be the formal
adjoint of $d$ with respect to the Euclidean product (\ref{2.4}).
Denote by $v$ the gradient of $f$ with respect to the given
Euclidean scalar product, then
\begin{align}\label{2.7}
v(Z)=-\sum_{\alpha=1}^{j}Z_{\alpha}e_{\alpha}
+\sum_{\alpha=j+1}^{n}Z_{\alpha}e_{\alpha}.
\end{align}
Set
\begin{align}\begin{split}\label{2.7a}
d_{T}=e^{-Tf}d \cdot e^{Tf}&=d+T\ df\wedge, \ \ \delta_{T}=e^{Tf}\delta\cdot e^{-Tf}
=\delta+T \ i_{v},
\\  D_{T, v}&=d_{T}+\delta_{T}=d+\delta+T\widehat{c}(v).
\end{split}\end{align} The Witten Laplacian on Euclidean space $V$
is by definition, the square of $D_{T, v}$.
Let $\Delta$ be the standard Laplacian on $V$, i.e.,
\begin{align}\label{2.8}
\Delta=-\sum^{n}_{\alpha=1}\frac{\partial^{2}}{\partial
Z^{2}_{\alpha}}. \end{align}
Let $e^{1},\ldots, e^{n}$
be the dual basis associated to $e_{1},\ldots, e_{n}$.
Then we have the following result \cite{Witten82},\ \cite[Prop.\,4.9]{Zhang01}.

\begin{prop}\label{t2.1}  The kernel of $D_{T, v}^2$ is of one dimension and is
spanned by
\begin{align}\label{2.9}\beta=e^{-\frac{T{|Z|}^{2}}{2}}e^{1} \wedge\ldots\wedge
e^{j}.\end{align} Moreover, all nonzero eigenvalues of $D_{T, v}^{2}$
are $\geqslant 2T$. \end{prop}

\begin{proof} We recall the proof for the reader's convenience.

For $e\in V$, let $\nabla_{e}$ denote the derivative along the vector $e$.
It is easy to calculate the square of $D_{T, v}$,
\begin{align}\begin{split}\label{2.10}
D_{T, v}^{2}&=\Delta+T^2{|Z|}^{2}+
T\sum_{\alpha=1}^{n}c(e_{\alpha})\widehat{c}(\nabla_{e_{\alpha}}v)
\\&=(\Delta+T^2{|Z|}^{2}-Tn)+T\sum_{\alpha=1}^{j}\big[1-c(e_{\alpha})
\widehat{c}(e_{\alpha})\big]
+T\sum_{\alpha=j+1}^{n}\big[1+c(e_{\alpha})\widehat{c}(e_{\alpha})\big]
\\&=(\Delta+T^2{|Z|}^{2}-Tn)+2T\big(\sum_{\alpha=1}^{j}i_{e_{\alpha}}
{e^{\alpha}}\wedge +\sum_{\alpha=j+1}^{n}{e^{\alpha}}\wedge
i_{e_{\alpha}}\big).\end{split}\end{align}
The operator
\begin{align}\label{2.11}\mathcal{L}_{T}=\Delta+T^2{|Z|}^{2}-Tn
\end{align} is the harmonic oscillator operator on $V$.
By \cite[Appendix E]{Ma07}, we know
that $\mathcal{L}_{T}$ is positive elliptic operator with
one dimensional kernel generated by $e^{-\frac{T{|Z|}^2}{2}}$.
Moreover, the nonzero eigenvalues of $\mathcal{L}_{T}$ are all
greater than $2T$.  It is also easy to verify that the linear
operator
\begin{align}\label{2.12}
\sum_{\alpha=1}^{j}i_{e_{\alpha}}{e^{\alpha}}\wedge
+\sum_{\alpha=j+1}^{n}{e^{\alpha}}\wedge
i_{e_{\alpha}}\end{align} is positive with
one dimensional kernel generated by
\begin{align}\label{2.13}
e^{1}\wedge\ldots\wedge e^{j}.
\end{align}
The proof of Proposition \ref{t2.1} is complete. \end{proof}

\subsection{The Witten Laplacian on $\mathbb{R}_{+}^{n}$}

Let $V_{1}$ be an $n-1$ dimensional real vector space endowed with an
Euclidean scalar product.  Let $V_{1}^{-}, V_{1}^{+}$ be two subspaces such
that $V_{1}=V_{1}^{-} \oplus V_{1}^{+}$ and $\textrm{dim} \ V_{1}^{-}=j$.
Let $e_{1}, \ldots, e_{n-1}$ denote an orthonormal basis on $V_{1}$ such that
$V_{1}^{-}$ is spanned by $e_{1}, \ldots, e_{j}$.
Set $V_{2}=V_{1}\times \mathbb{R}_{+}$. Denote by $e_{n}$ the oriented basis of
$\mathbb{R}_{+}$ with unit length.

Let $f$ be a smooth function on $V_{2}$ given by:
\begin{align}\label{3.13}
f(Z)=f(0)-\frac{{\left\vert Z^{-}\right\vert}^2}{2}+
\frac{{\left\vert Z^{+}\right\vert}^2}{2}+Z_{n}, \end{align} where
$Z^{-}=(Z_{1}, \ldots, Z_{j}), Z^{+}=(Z_{j+1}, \ldots,
Z_{n-1}), (Z^{-}, Z^{+})$ denotes the coordinate functions on $V_{1}$
corresponding to the decomposition $V_{1}=V_{1}^{-}\oplus V_{1}^{+}$ and $Z_{n}\geqslant 0$
denotes the coordinate function on $\mathbb{R}_{+}$. Set $Z'=(Z^{-},Z^{+})$.

Let $L^{2}(\Lambda V_{2}^{\ast})$ be the set of the square integrable sections of $\Lambda V^{\ast}_{2}$.
We define an inner product in $L^{2}(\Lambda V_{2}^{\ast})$ similarly to (\ref{2.4}).
Let $C(V_{2})$ be the Clifford algebra of $V_{2}$. We still denote by $c(e), \hat{c}(e)$ the
Clifford operators on $\Lambda V^{\ast}_{2}$ for $e\in V_{2}$.

The gradient vector field $v_{1}$ of $f$ with respect to the given
Euclidean scalar product is now given as
\begin{align}
v_{1}(Z)=-\sum_{\alpha=1}^{j}Z_{\alpha}e_{\alpha}
+\sum_{\alpha=j+1}^{n-1}Z_{\alpha}e_{\alpha}+e_{n}.
\end{align}

Set
\begin{align}\begin{split}\label{3.15}
d_{T}=e^{-Tf}d \cdot e^{Tf}&=d+T\ df\wedge, \ \ \delta_{T}=e^{Tf}\delta\cdot e^{-Tf}
=\delta+T \ i_{v_{1}},
\\  D_{T, v_{1}}&=d_{T}+\delta_{T}=d+\delta+T\widehat{c}(v_{1}).
\end{split}\end{align}
To calculate explicitly the kernel of $D_{T, v_{1}}^{2}$ and,
more generally, to study its spectrum, we need to consider boundary conditions on
$\Lambda V_{2}^{\ast}$, i.e., to specify the domain of a self-adjoint extension of
$D_{T, v_{1}}^{2}$. We follow here \cite[\S\,3.5]{Ma07}. Denote by $\Omega(V_{2})$ the space of
smooth sections of $\Lambda V_{2}^{\ast}$ on $V_{2}$.
Define the domain of the weak maximal extension of $d_T$ by
\begin{equation}\label{1.11}
\Dom(d_T)=\big\{w\in L^2(\Lambda V_{2}^{\ast}):d_Tw\in L^2(\Lambda V_{2}^{\ast})\big\}
\end{equation}
where $d_Tw$ is calculated in the sense of distributions. We denote by $d_T^*$ the Hilbert
space adjoint of $d_T$. Let $e^n$ be the dual basis associated to $e_n$. Every $w\in \Omega(V_{2})$
has a natural decomposition into the norm and the tangent components along $V_{1}$,
\begin{equation}
w=w_{\textup{tan}}+w_{\textup{norm}},
\end{equation}
where $w_{\textup{tan}}$ does not contain the factor $e^{n}$.
Integration by parts shows
\begin{equation}
\begin{split}
\Dom&(d_T^*)\cap\Omega(V_{2})=\big\{w\in \Omega(V_{2}):w_{\textup{norm}}=0\big\}\,,\\
&d_T^*w=\delta_Tw\:\:\text{for $w\in\Dom(d_T^*)\cap\Omega(V_{2})$}\,.
\end{split}
\end{equation}
We define then the Gaffney extension of $D^{2}_{T, v_{1}}$ by
\begin{equation}\label{1.22}
\begin{split}
\Dom(D^{2}_{T, v_{1}})&=\big\{w\in\Dom(d_T)\cap\Dom(d_T^*):d_Tw\in\Dom(d_T^*)\,,\, d_T^*w\in\Dom(d_T)\big\}\,,\\
&D^{2}_{T, v_{1}}w=d_Td_T^*w+d_T^*d_Tw\:\:\text{for $w\in\Dom(D^{2}_{T, v_{1}})$}\,.
\end{split}
\end{equation}
This is a self-adjoint operator, see \cite[Prop.\,3.1.2]{Ma07}.
The smooth forms in the domain of $D^{2}_{T, v_{1}}$ satisfy the following boundary conditions
\begin{align}
\textup{Dom}(D_{T, v_{1}}^{2})\cap\Omega(V_{2})=\Big\{ w\in \Omega(V_{2}),
\begin{split} w_{\textup{norm}}&=0 \ \
\\ \big(d_{T} w\big)_{\textup{norm}}&
=0 \end{split}
\ \ \textup{on}\ V_{1}\Big\}.
\end{align}
If we rewrite $w\in \Omega(V_{2})$ as
\begin{align}
w(Z', Z_{n})=w_{1}(Z', Z_{n})+e^{n}\wedge w_{2}(Z', Z_{n}), \ Z' \in V_{1}, Z_{n}\in \mathbb{R}_{+},
\end{align}
\noindent where $w_{1}$ does not contain the factor $e^{n}$, then we have
\begin{lemma}
The boundary conditions
\begin{align}
w_{\textup{norm}}=0, \ (d_{T}w)_{\textup{norm}}=0\ \textup{on}\ V_{1}
\end{align}
are equivalent to
\begin{align}\label{1.22a}
\frac{\partial w_{1}}{\partial Z_{n}}(Z', 0)+Tw_{1}(Z', 0)=0 \ \textup{and}\ w_{2}(Z', 0)=0.
\end{align}
\end{lemma}
\begin{proof}
The proof is straightforward and is left to the reader.
\end{proof}
Let $e^{1},\ldots, e^{n-1}$
be the dual basis of $e_{1},\ldots, e_{n-1}$.
We denote by $\Delta' \ (\textup{resp}. \ \Delta)$ the
standard Laplacian on $V_{1} \ (\textup{resp}. \ V_{2})$,  i.e.,
\begin{align}
\Delta'=-\sum^{n-1}_{\alpha=1}\frac{\partial^{2}}{\partial
Z^{2}_{\alpha}} \ \ \  \Big(\textup{resp}. \ \Delta=-\sum^{n}_{\alpha=1}\frac{\partial^{2}}{\partial
Z^{2}_{\alpha}}\Big). \end{align}
\begin{prop} \label{t2.3}
The kernel of $D_{T, v_{1}}^2$, with the domain given as in \eqref{1.22}, is
one dimensional and is spanned by
\begin{align}
\beta_{T, v_{1}}=e^{-\frac{T}{2}|Z'|^{2}-TZ_{n}}
e^{1} \wedge\ldots\wedge e^{j}.\end{align}
Moreover, all nonzero eigenvalues of $D_{T, v_{1}}^{2}$
are $\geqslant 2T$.
\end{prop}
\begin{proof} We adapt the proof from \cite[Prop.\,4.9]{Zhang01}.
We denote by $\nabla_{e}$ the derivative
along a vector $e\in V_{2}$. It is easy to calculate the square of $D_{T, v_{1}}$:
\begin{align}\label{1.25}
D_{T, v_{1}}^{2}=&\Delta+T^2\Big({|Z'|}^{2}+1\Big)+
T\sum_{\alpha=1}^{n}c(e_{\alpha})\widehat{c}\big(\nabla_{e_{\alpha}}v_{1}\big)
\nonumber \\=&
\Big(\Delta'+T^2{|Z'|}^{2}-T(n-1)\Big)+\Big(-\frac{\partial^{2}}{\partial Z_{n}^{2}}+T^{2}\Big)
\nonumber \\&
+T\sum_{\alpha=1}^{j}\big[1-c(e_{\alpha})
\widehat{c}(e_{\alpha})\big]
+T\sum_{\alpha=j+1}^{n-1}\big[1+c(e_{\alpha})\widehat{c}(e_{\alpha})\big]
\\=&
\Big(\Delta'+T^2{|Z'|}^{2}-T(n-1)\Big)+\Big(-\frac{\partial^{2}}{\partial Z_{n}^{2}}+T^{2}\Big)
\nonumber \\&
+2T\big(\sum_{\alpha=1}^{j}i_{e_{\alpha}}
{e^{\alpha}}\wedge +\sum_{\alpha=j+1}^{n-1}{e^{\alpha}}\wedge
i_{e_{\alpha}}\big).\nonumber
\end{align}

Since the three operators in parentheses on the right side of the third equality
in (\ref{1.25}) commute with each other, we can calculate the kernels of
$D^{2}_{T, v_{1}}$ via the method of separating variables.
As in Proposition \ref{t2.1}, the kernel of the operator
\begin{align}
\mathcal{L}'_{T}=\Delta'+T^2{|Z'|}^{2}-T(n-1)\end{align}
is one dimensional and generated by $e^{-\frac{T{|Z'|}^2}{2}}$.
Besides, its nonzero eigenvalues are all
greater than $2T$.  It is easy to verify that the linear
operator
\begin{align}\label{lin}
\sum_{\alpha=1}^{j}i_{e_{\alpha}}{e^{\alpha}}\wedge
+\sum_{\alpha=j+1}^{n-1}{e^{\alpha}}\wedge
i_{e_{\alpha}}\end{align}
restricted to $\textup{Dom}(D_{T, v_{1}}^{2})$ is positive and has one
dimensional kernel generated by
\begin{align}\label{1.28}
e^{1}\wedge\ldots\wedge e^{j}.
\end{align}
Therefore, the elements of the kernel of $D_{T, v_{1}}^{2}$ take
the form
\begin{align}
g(Z_{n})e^{-\frac{T{|Z'|}^2}{2}}e^{1}\wedge\ldots\wedge e^{j},
\end{align} where $g(Z_{n})$ is a smooth function in $Z_{n}$.
Combining (\ref{1.22a}) and (\ref{1.25}), we find that the smooth function $g(Z_{n})$ satisfies
the following conditions:
\begin{align}\label{2.25}
g'(0)+Tg(0)=0, \ \ \ -g''(Z_{n})+T^{2}g(Z_{n})=0.
\end{align}
Thus, $g(Z_{n})=e^{-TZ_{n}}$ up to multiplicative constant.
Set
\begin{align}
R_{T}=-\frac{\partial^{2}}{\partial Z_{n}^{2}}+T^{2}
\end{align} with domain given by
\begin{align}
\textup{Dom}(R_{T})=\big\{h\in C^{\infty}_{0}(\mathbb{R}_{+}), \ h'(0)+Th(0)=0 \big\}.
\end{align}
One verifies directly via integration by parts that $R_{T}$ is positive on $\Dom(R_{T})$,
so $R_{T}$ has a self-adjoint positive extension (the Friedrichs extension) still denoted by $R_{T}$.
Denote by $\lambda_{1}(P)$ the first nonzero eigenvalues of a positive operator $P$.
By min-max principle \cite[(C.3.3)]{Ma07} and the fact that $R_T$ and \eqref{lin} are positive operators, we deduce that
\begin{align}
\lambda_{1}(D^{2}_{T, v_{1}})\geqslant \lambda_{1}(\mathcal{L}'_{T})= 2T.
\end{align}
The proof of Proposition \ref{t2.3} is complete. \end{proof}

\section{The Witten instanton complex}\label{s2}
Let $g^{TM}$ be a metric on $TM$ such that if $p\in C^{j}(f)\cup C^{j}_{-}(f|_{\partial M})$,
in the coordinates $x=(x_{1}, \ldots, x_{n})$ in the conditions (2) and (3)
of Section \ref{s5},
\begin{align}
g^{TM}=\sum_{\alpha=1}^{n}dx^{2}_{\alpha} \ \  \textup{on}\ U_{p}.
\end{align}
Let $\nabla^{TM}$ be the Levi-Civita
connection associated to the metic $g^{TM}$,
and let $dv_{M}$ be the density (or Riemannian volume form) on $M$, i.e., $dv_{M}$ is a smooth
section of the line bundle $\Lambda^{n}(T^{\ast}M)\otimes o(TM)$ (cf. \cite[p.\,29]{Berline04},
\cite[p.\,88]{Bott82}). Denote by $\Omega^{i}(M)$ the space of smooth differential $i$-forms on $M$.
Set $\Omega(M)=\oplus^{n}_{i=0}\Omega^{i}(M)$.
We denote by $L^{2}\Omega(M)$ the space of square integrable sections of $\Lambda(T^{\ast}M)$
over $M$. For $w_{1}, w_{2}\in L^{2}\Omega(M)$, set
\begin{align}\label{1.2}
\big\langle w_{1}, w_{2}\big\rangle=\int_{M}\big\langle w_{1}, w_{2} \big\rangle(x) dv_{M}(x).
\end{align} We denote by $\|\cdot\|$ the norm on $L^{2}\Omega(M)$
induced by (\ref{1.2}).

Let $d$ be the exterior differential derivative
on $\Omega(M)$, and let $\delta$ be the formal adjoint of $d$ with
respect to the metric (\ref{1.2}).  Set
\begin{align}\label{1.4}
d_{T}=e^{-Tf}d\cdot e^{Tf}, \ \
\delta_{T}=e^{Tf}\delta\cdot e^{-Tf}.
\end{align}
The deformed de Rham operator $D_{T}$ is given by
\begin{align}\label{1.4b}
D_{T}=d_{T}+\delta_{T}.
\end{align}
The Witten Laplacian on manifolds is defined by
\begin{align}\label{1.4a}
D_{T}^{2}=(d_{T}+\delta_{T})^{2}=d_{T}\delta_{T}+\delta_{T}d_{T}.
\end{align}
Define the domain of the weak maximal extension of $d_T$ by
\begin{equation}
\Dom(d_T)=\big\{w\in L^2\Omega(M),d_Tw\in L^2\Omega(M)\big\}
\end{equation}
where $d_Tw$ is calculated in the sense of distributions.
We denote by $d_T^*$ the Hilbert space adjoint of $d_T$.
Every smooth differential form $w$
has a natural decomposition into the norm and the tangent components along $\partial M$,
\begin{align}\label{1.5}
w=w_{\textup{tan}}+w_{\textup{norm}},
\end{align} where $w_{\textup{tan}}$ does not contain the factor ${\nu}$.
Integration by parts shows
\begin{equation}\label{1.6}
\begin{split}
\Dom&(d_T^*)\cap\Omega(M)=\big\{w\in \Omega(M):w_{\textup{norm}}=0\big\}\,,\\
&d_T^*w=\delta_Tw\:\:\text{for $w\in\Dom(d_T^*)\cap\Omega(M)$}\,.
\end{split}
\end{equation}
By \eqref{1.6}, the domain of the extension $D_T=d_T+d_T^*$ of deformed de Rham  operator is:
\begin{align}\label{1.7}
\textup{Dom}(D_{T})\cap\Omega(M)=\big\{w \in \Omega(M), w_{\textup{norm}}=0 \ \ \textup{on} \ \partial M\big\}.
\end{align}
We define the self-adjoint extension of $D_{T}^{2}$ as in \eqref{1.22} by $D_{T}^{2}=d_Td_T^*+d_T^*d_T$.
Then
\begin{align}\label{1.8}
\textup{Dom}(D^{2}_{T})\cap\Omega(M)=\Big\{ w\in \Omega(M),
\begin{split} w_{\textup{norm}}&=0 \ \
\\   \big(d_{T} w\big)_{\textup{norm}}&
=0  \end{split} \ \ \textup{on} \ \partial M  \Big\}.
\end{align}
%
%
Following the argument of \cite[Prop.\,5.5]{Zhang01}, one easily gets Morse inequalities (\ref{1.1})
granted the following Proposition holds.
 \begin{prop}\label{t1.2}
  For any $C_{0}>0$, there exists $T_{0}>0$ such that when $T\geqslant T_{0}$, the number
  of eigenvalues in $[0, C_{0})$ of $D^{2}_{T}\big|_{\textup{Dom}(D^{2}_{T})\cap \Omega^{j}(M)}$ equals
 $c_{j}+p_{j}$. \end{prop}
 We postpone the proof of Proposition \ref{t1.2} to Section \ref{s7}. We prove now Corollary \ref{t1.1} by
 using Proposition \ref{t1.2}.

\begin{proof}[Proof of Corollary \ref{t1.1}]
Let $F^{C_{0}}_{T, j}$ denote the
$(c_{j}+p_{j})$-dimensional vector space generated by the eigenspaces of
$D^{2}_{T}|_{\textup{Dom}(D^{2}_{T})\cap \Omega^{j}(M)}$ associated to the eigenvalues lying in [0, $C_{0}$).
Since $d_{T}D^{2}_{T}=D^{2}_{T}d_{T}$, one verifies directly that $d_{T}(F^{C_{0}}_{T,j})\subset F_{T,j+1}^{C_{0}}$.
Then we have the following complex:
\begin{align}\label{1.8a}
(F^{C_{0}}_{T,\bullet}, d_{T}):0\longrightarrow F^{C_{0}}_{T, 0}
\longrightarrow F^{C_{0}}_{T, 1} \longrightarrow\ldots \rightarrow
F^{C_{0}}_{T, n}\longrightarrow  0.
\end{align}
By Hodge Theorem in the finite dimensional case, the $j$-th
cohomology group of the above complex is isomorphic to
$\textup{Ker} \big(D^{2}_{T}|_{\textup{Dom}(D^{2}_{T})\cap \Omega^{j}(M)}\big)$,
which is again by Hodge Theorem isomorphic to
the $j$-th cohomology group of the deformed de Rham
complex $(\Omega^{\bullet}(M), d_{T})$. It is a consequence of (\ref{1.4}) that
the $j$-th cohomology group of the deformed de Rham
complex $(\Omega^{\bullet}(M)$, $d_{T})$ is isomorphic to
the $j$-th cohomology group of the de Rham complex $(\Omega^{\bullet}(M), d)$.
Then the inequalities (\ref{1.1}) follow from standard algebraic techniques
(\cite[Lemma 3.2.12]{Ma07}).
 \end{proof}

 \begin{rem} We can also obtain the inequalities (\ref{1.1}) immediately
by combining Theorem \ref{t1.0} and \cite[Th.\,A]{Laudenbach10}.
\end{rem}

\begin{proof}[Proof of Corollary \ref{t1.2a}] By Poincar\'e duality theorem for non-orientable manifolds,
\begin{align}\label{1.8b}
H_{dR}^{\bullet}(M, \mathbb{R})\simeq H^{n-\bullet}_{dR}(M, \partial M; o(TM)).
\end{align}
Then the inequalities (\ref{1.1a}) follow immediately from the inequalities (\ref{1.1}) and
the isomorphism (\ref{1.8b}).
 \end{proof}

\section{Proof of Proposition \ref{t1.2} and Theorem \ref{t1.0}}\label{s6}
The organization of the Section is as follows.
In Section \ref{s2a}, we obtain a basic estimate for the deformed de Rham operator
which allows to localize our problem
(i.e., study the eigenvectors with small eigenvalues of the Witten Laplacian) to some
neighborhood of critical points in $C(f)$ and $C_{-}(f|_{\partial M})$.
Section \ref{s3} is devoted to the local behavior of the Witten Laplacian around
critical points in $C(f)$ and $C_{-}(f|_{\partial M})$. In Section \ref{s4}, we get a decomposition
of the deformed de Rham operator and establish estimates of its components.
We also prove Theorem \ref{t1.3} and Theorem \ref{t1.4} there. Section \ref{s7}
is devoted to the proof of Proposition \ref{t1.2} and Theorem \ref{t1.0}.

\subsection{Localization of the lower part of the spectrum of the
Witten Laplacian}\label{s2a}
Choose $a$ small enough such that all $B(p, a)$'s  are disjoint for
$p\in C(f)\cup C_{-}(f|_{\partial M})$, and
each $B(p, a)$ lies in the interior of $M$ for $p\in C(f)$. Denote by $U$ the union of
all $B(p, a)$'s for $p\in C(f)\cup C_{-}(f|_{\partial M})$.
\begin{prop}\label{t4.2}
There exist constants $a_{4}>0, T_{3}>0$ such that for any
$s\in\Dom(D_{T})$ with $\supp(s)\subset M\backslash U$ and
$T\geqslant T_{3}$, we have
\begin{align}\label{2.0}
\big\|D_{T}s\big\|\geqslant a_{4}T\big\|s\big\|.
\end{align}
\end{prop}
\begin{proof}
We adapt our proof from \cite[pp.\,29-30]{Peutrec10}. Note that the Green formula holds also on
non-orientable manifolds, see \cite[Chapter 2,\,Th.\,2.1]{Taylor96}.
It is a consequence of (\ref{1.4a}) and the Green's formula that
\begin{align}\label{2.1}
\big\|D_{T}s\big\|^{2}=\big\langle d_{T}s, d_{T}s\big\rangle
+\big\langle \delta_{T}s, \delta_{T} s\big\rangle.
\end{align}
Let $\nabla f$ denote the gradient field of $f$ with respect to
the metric $g^{TM}$. Since $s_{\textup{norm}}=0$, one verifies directly from the Green's formula that
\begin{align}\begin{split}\label{2.2}
\big\langle d_{T}s, d_{T}s\big\rangle
+\big\langle \delta_{T}s, \delta_{T} s\big\rangle=&\big\langle ds, ds\big\rangle
+\big\langle \delta s, \delta s\big\rangle+T^{2}\big\langle |\nabla f|^{2}s, s\big\rangle
\\&+T\big\langle Qs, s \big\rangle+
T\int_{\partial M}\big\langle (\nu f)s_{\textup{tan}}, s_{\textup{tan}}\big\rangle d_{\partial M}(x),
\end{split}\end{align}
where $dv_{\partial M}$ denotes the density on $\partial M$ induced by $dv_{M}$,
$\nu$ outward normal field and $Q$ is an endomorphism of $\Omega(M)$ given by
\begin{align}
Q=\sum_{i=1}^{n}c(e_{i})\hat{c}\big(\nabla_{e_{i}}^{TM}(\nabla f)\big).
\end{align}
If $\supp(s)\cap \partial M=\emptyset$, then
\begin{align}\label{2.2a}
\int_{\partial M}\big\langle (\nu f)s_{\textup{tan}}, s_{\textup{tan}}\big\rangle d_{\partial M}(x)=0.
\end{align}
Then (\ref{2.0}) follows immediately from (\ref{2.1}), (\ref{2.2}) and (\ref{2.2a}).
Suppose now $\supp(s)\cap \partial M\neq \emptyset$. Clearly,
\begin{align}
-(vf)(x)<|\nabla f|(x), \ \  \textup{for any}\ x\in \ \supp(s)\cap \partial M.
\end{align}
Choose $\eta>0$ small enough so that
\begin{align}\label{2.2b}
-(vf)(x)<(1-\eta)|\nabla f|(x), \ \  \textup{for any}\ x\in \ \supp(s)\cap \partial M.
\end{align}
Locally it is possible to construct a function $\tilde{f}$
such that
\begin{align}\begin{split}\label{2.2c}
\big|\nabla\tilde{f}\big|&=\big|\nabla f\big|, \ \ \textup{in}\ \ \supp(s);
\\
\big|\nabla\tilde{f}\big|&=-\nu \tilde{f}, \ \ \textup{in}\ \ \supp(s)\cap \partial M.
\end{split}\end{align}
Let $\tilde{V}$ be an open neighborhood of $\partial M$ such that the equations (\ref{2.2c})
are solvable on $V_{r}\cap \supp(s)$ with $V_{r}=\tilde{V}\times [0, r)$. By a partition of
unity argument, we now assume that $s\in \Dom(D_{T})$
with $\supp(s)\subset V_{r}$ and the solution $\tilde{f}$ of the equations (\ref{2.2c})
 is a smooth function on $M$.
Then (\ref{2.2b}) and (\ref{2.2c}) imply
\begin{align}\label{2.2d}
T\int_{\partial M}\big\langle (\nu f)s_{\textup{tan}}, s_{\textup{tan}}\big\rangle d_{\partial M}(x)
 \geqslant
T(1-\eta)\int_{\partial M}\big\langle (\nu \tilde{f}) s_{\textup{tan}}, s_{\textup{tan}}\big\rangle d_{\partial M}(x).
\end{align}
Applying the equality (\ref{2.2}) to the smooth function $\tilde{f}$, we find
\begin{align}\begin{split}\label{2.2e}
&T\int_{\partial M}\big\langle (\nu f)s_{\textup{tan}}, s_{\textup{tan}}\big\rangle d_{\partial M}(x)
 \\ \geqslant &
-(1-\eta)\Big[\big\langle ds, ds \big\rangle+\big\langle \delta s, \delta s \big\rangle
+T^{2}\big\langle |\nabla f|^{2}s, s\big\rangle +C_{1}T\|s\|^{2}\Big],
\end{split}\end{align}
where $C_{1}>0$ is independent of $s$. Substituting (\ref{2.2e}) into (\ref{2.2}), we obtain the estimate
\eqref{2.0}. The proof of Proposition \ref{t4.2} is complete.
\end{proof}
In view of Proposition \ref{t4.2}, the eigenvectors with small eigenvalues of the
Witten Laplacian $D_T^2$ ``concentrate" for $T$ large around the critical points in $C(f)$ and $C_{-}(f|_{\partial M})$.

\subsection{Local behavior of the Witten Laplacian around critical points
in $C(f)$ and $C_{-}(f|_{\partial M})$}
\label{s3} If $p\in C^{j}(f)$, then there exists a
coordinate system $(x, U_{p})$ such that for any $x\in U_{p}$,
\begin{align}\label{3.3a}
f(x)=f(p)-\frac{x^{2}_{1}}{2}-\ldots-\frac{x^{2}_{j}}{2}+\frac{x^{2}_{j+1}}{2}+\ldots+ \frac{x^{2}_{n}}{2}.
\end{align}
Over $U_{p}$,  set $e_{k}=\frac{\partial}{\partial x_{k}}$ for $k=1, \ldots, n$.
Then the dual basis $e^{k}=dx_{k}$ for all $k$.
From (\ref{2.3}), (\ref{2.7a}), (\ref{1.4b}) and (\ref{3.3a}), we have
\begin{align}\label{3.4}
D_{T}|_{U_{p}}=D_{T, v}|_{U_{p}}.
\end{align}
Set
\begin{align}\begin{split}\label{3.5}
\mathcal{L}_{T}=&-\sum^{n}_{\alpha=1}\frac{\partial^{2}}{\partial x^{2}_{\alpha}}
+T\sum^{n}_{\alpha=1}x_{\alpha}^{2}-Tn,
\\ \mathcal{K}_{T}=&2T\Big(\sum_{\alpha=1}^{j}i_{e_{\alpha}}{e^{\alpha}}\wedge
+\sum_{\alpha=j+1}^{n}{e^{\alpha}}\wedge
i_{e_{\alpha}}\Big)
\end{split}\end{align}
Then (\ref{2.10}) and (\ref{3.4}) imply that
\begin{align}\label{3.6}
D^{2}_{T}=\mathcal{L}_{T}+\mathcal{K}_{T}
\end{align}
holds throughout $U_{p}$. Denote by $L^2\Omega(\mathbb{R}^{n})$ the space of square integrable differential forms
on $\mathbb{R}^{n}$. By Proposition \ref{t2.1}, we have
\begin{prop}\label{t4.1a}
For any $T>0$, the operator $\mathcal{L}_{T}+\mathcal{K}_{T}$ acting on
$L^2\Omega(\mathbb{R}^{n})$ is an essentially self-adjoint positive operator. Its kernel is one dimensional and is spanned by
\begin{align}
\beta_{T}=e^{-\frac{T}{2}|x|^{2}}e^{1}\wedge \ldots
\wedge e^{j}.
\end{align}
Moreover, all nonzero eigenvalues of $\mathcal{L}_{T}+\mathcal{K}_{T}$ are $\geqslant 2T$.
\end{prop}
If $p\in C^{j}_{-}(f|_{\partial M})$, then there exists
a coordinate system $(x, U_{p})$ such that for any $x\in U_{p}$,
\begin{align}\label{3.8}
f(x)=f(p)-\frac{x^{2}_{1}}{2}-\ldots-\frac{x^{2}_{j}}{2}+\frac{x^{2}_{j+1}}{2}+\ldots \frac{x^{2}_{n-1}}{2}+x_{n}.
\end{align}
On $U_{p}$, set $e_{k}=\frac{\partial}{\partial x_{k}}$ for $k=1,\ldots, n$. Then
the dual basis  $e^{k}=dx_{k}$ for all $k$.
In view of (\ref{3.13}), (\ref{3.15}), (\ref{1.4b}) and (\ref{3.8}), we have
\begin{align}\label{3.9}
D_{T}|_{U_{p}}=D_{T, v_{1}}|_{U_{p}}.
\end{align}
Set
\begin{align}\begin{split}
\mathcal{L}'_{T}=&-\sum^{n}_{\alpha=1}\frac{\partial^{2}}{\partial x^{2}_{\alpha}}
+T\sum^{n-1}_{\alpha=1}x_{\alpha}^{2}-T(n-1)+T^{2},
\\ \mathcal{K}'_{T}=&2T\Big(\sum_{\alpha=1}^{j}i_{e_{\alpha}}{e^{\alpha}}\wedge
+\sum_{\alpha=j+1}^{n-1}{e^{\alpha}}\wedge
i_{e_{\alpha}}\Big).
\end{split}\end{align}
Combining (\ref{1.25}) and (\ref{3.9}), we find that on $U_{p}$,
\begin{align}\label{5.19}
D^{2}_{T}=\mathcal{L}'_{T}+\mathcal{K}'_{T}.
\end{align}
Let $\Omega(\mathbb{R}_{+}^{n})$ be the space of smooth differential sections
on $\mathbb{R}_{+}^{n}$. Consider the self-adjoint extension of \eqref{5.19}
defined as in \eqref{1.11}-\eqref{1.22}. Thus 
\begin{align}\label{3.12}
\Dom\big(\mathcal{L}'_{T}+\mathcal{K}'_{T}\big)\cap\Omega(\mathbb{R}_{+}^{n})=
\Big\{ w\in\Omega(\mathbb{R}_{+}^{n}),
\begin{split} w_{\textup{norm}}&=0 \ \
\\   \big(d_{T} w\big)_{\textup{norm}}&
=0  \end{split} \ \ \textup{on} \ \partial\mathbb{R}^{n}_+\Big\}.
\end{align}
By Proposition \ref{t2.3}, we have
\begin{prop}\label{t4.1b}
For any $T>0$, the self-adjoint extension of $\mathcal{L}'_{T}+\mathcal{K}'_{T}$, given as in \eqref{1.11}-\eqref{1.22}
is a positive operator. Its kernel is one dimensional and is spanned by
\begin{align}
\xi_{T}=e^{-\frac{T}{2}|x'|^{2}-Tx_{n}}e^{1}\wedge \ldots
\wedge e^{j}.
\end{align}
Moreover, all nonzero eigenvalues of $\mathcal{L}'_{T}+\mathcal{K}'_{T}$ are $\geqslant 2T$.
\end{prop}
\subsection{A decomposition of the deformed de Rham operator $D_{T}$} \label{s4}
Let $\gamma: \mathbb{R}\rightarrow [0,1]$ be a smooth cut-off function such that
$\gamma(x)=1$ if $|x|\leqslant a$ and that $\gamma(x)=0$ if $|x|\geqslant 2a$.
For any $p\in C^{j}(f)$ and $q\in C_{-}^{j}(f|_{\partial M})$, set
\begin{align}\begin{split}
\alpha_{p,T}=&\int_{U_{p}}\gamma(|x|)^{2} e^{-T|x|^{2}}dx_{1}\wedge \ldots \wedge dx_{n},
\\ \alpha_{q,T}=&\int_{U_{q}}\gamma(|x|)^{2} e^{-T|x'|^{2}-2Tx_{n}}dx_{1}\wedge \ldots \wedge dx_{n}.
\end{split}\end{align}
Clearly, there exists $c>0$ such that as $T\rightarrow +\infty,$
\begin{align}\begin{split}
\alpha_{p, T}=&\big(\frac{\pi}{T}\big)^{\frac{n}{2}}+O(e^{-cT}),
\\ \alpha_{q, T}=&\frac{1}{2T}\big(\frac{\pi}{T}\big)^{\frac{n-1}{2}}+O(e^{-cT}).
\end{split}\end{align}
Set
\begin{align}\begin{split}\label{5.2}
\rho_{p,T}=& \frac{\gamma(|x|)}{\sqrt{\alpha_{p,T}}}
e^{-\frac{T}{2}|x|^{2}}dx_{1}\wedge \ldots \wedge dx_{j},
\\ \rho_{q,T}=& \frac{\gamma(|x|)}{\sqrt{\alpha_{q,T}}}
e^{-\frac{T}{2}|x'|^{2}-Tx_{n}} dx_{1}\wedge \ldots \wedge dx_{j}.
\end{split}\end{align}
Let $E^{j}_{T}$ be the direct sum of the vector spaces generated by all
$\rho_{p, T}$'s and $\rho_{q, T}$'s with $p\in C^{j}(f)$ and $q\in C^{j}_{-}(f|_{\partial M})$.
Set $E_{T}=\oplus^{n}_{j=0} E^{j}_{T}$.
Clearly,
\begin{align}
\text{dim}\ E_{T}=\sum_{j=0}^{n}\big(c_{j}+p_{j}\big).
\end{align}
Take $E^{\bot}_{T}$ as the orthogonal complement of $E_{T}$ in $\Dom(D_{T})$,
then $\Dom(D_{T})$ has an orthogonal splitting:
\begin{align}\label{5.5}
\Dom(D_{T})=E_{T}\oplus E^{\bot}_{T}.
\end{align}
Let $p_{1}, p^{\bot}_{1}$ denote the orthogonal projections from $\Dom(D_{T})$
onto $E_{T}$ and $E^{\bot}_{T}$, respectively. Also we have another orthogonal splitting
about $E_{T}$ in $L^{2}\Omega(M)$:
\begin{align}\label{5.5a}
L^{2}\Omega(M)=E_{T}\oplus F_{T},
\end{align}
where $F_{T}$ is the orthogonal complement of $E_{T}$ in $L^{2}\Omega(M)$.
Then $E^{\bot}_{T}\subset F_{T}$.
Denote by $p_{2}, p^{\bot}_{2}$ the orthogonal projections from $L^{2}\Omega(M)$
onto $E_{T}$ and $F_{T}$, respectively.
Following Bismut-Lebeau \cite[\S 9]{Bismut74}, we decompose the deformed de Rham
operator $D_{T}$ according to the splittings (\ref{5.5}) and (\ref{5.5a}):
\begin{align}\begin{split}\label{5.6}
D_{T,1}=p_{2}D_{T}p_{1}, & \ \  D_{T,2}=p_{2}D_{T}p^{\bot}_{1},
\\ D_{T,3}=p^{\bot}_{2}D_{T}p_{1} &  \ \ D_{T,4}=p^{\bot}_{2}D_{T}p^{\bot}_{1}.
\end{split}\end{align} Then
\begin{align}\label{5.6a}
D_{T}=D_{T, 1}+D_{T, 2}+D_{T, 3}+D_{T, 4}.
\end{align}
Denote by ${\bf H}^{1}(M)$ the first Sobolev space with respect to a (fixed) Sobolev norm on $\Omega(M)$.
The analogues of the estimates \cite[Prop.\,5.6]{Zhang01} still hold for the operators $D_{T,j}$\,:
\begin{prop}\label{t6.1}
(1) For any $T>0$,
\begin{align}\label{5.11a}
D_{T,1}=0;
\end{align}
(2) There exist positive constants $b_{1}, b_{2}$ and $T_{4}$ such that for any $s\in E^{\bot}_{T}\cap {\bf H}^{1}(M),
s'\in E_{T}$ and any $T\geqslant T_{4}$, one has
\begin{align}\begin{split}\label{5.11}
\big\|D_{T,2}s\big\| &\leqslant b_{1}e^{-b_{2}T}\big\|s\big\|,
\\ \big\|D_{T,3}s'\big\| &\leqslant b_{1}e^{-b_{2}T}\big\|s'\big\|.
\end{split}\end{align}
(3) There exists constant $b_{3}>0$ and $T_{5}>0$, such that for any $s\in E^{\bot}_{T}\cap {\bf H}^{1}(M)$
and any $T\geqslant T_{5}$, one has
\begin{align}
\big\|D_{T, 4}s\big\| &\geqslant b_{3}T\big\|s\big\|.
\end{align}
\end{prop}
\begin{proof} The proof is similar to \cite[Prop.\,5.6]{Zhang01}
except for the estimate of $D_{T, 3}$.
Compared to \cite[(9.17)]{Bismut74} and \cite[(5.19)]{Zhang01}, the operator
$D_{T, 3}$ is no longer the formal adjoint of $D_{T, 2}$ due to the fact that the image of $D_{T}$ acting on
$\Dom(D_{T})$ does not necessarily lie in $\Dom(D_{T})$. We prove
the estimate of $D_{T, 3}$ directly.
From (\ref{5.2}), Proposition \ref{t4.1a} and Proposition \ref{t4.1b}, we have
\begin{align}\label{5.13}
D_{T}(\rho_{p,T})= \frac{c(\nabla \gamma)}{\sqrt{\alpha_{p,T}}}
e^{-\frac{T}{2}|x|^{2}}dx_{1}\wedge \ldots \wedge dx_{j}
\end{align}
and
\begin{align}\label{5.14}
D_{T}(\rho_{q,T})= \frac{c(\nabla \gamma)}{\sqrt{\beta_{q,T}}}
e^{-\frac{T}{2}|x'|^{2}-Tx_{n}} dx_{1}\wedge \ldots \wedge dx_{j},
\end{align}
where $c(\nabla \gamma)$ denotes the endomorphism on $\Omega(M)$ given as
\begin{align}
c(\nabla \gamma)=\sum_{i=1}^{n}(e_{i}\gamma) c(e_{i}).
\end{align}
Then (\ref{5.13}) and (\ref{5.14}) imply the second inequality in (\ref{5.11}).
The rest of the proof is similar to \cite[Prop.\,5.6]{Zhang01}.
\end{proof}
\begin{proof}[Proof of Theorem \ref{t1.3} and Theorem \ref{t1.4}]
By the min-max principle \cite[(C.3.3)]{Ma07},
\begin{align}\label{5.15}
\lambda_{k}(T)=\inf_{\stackrel{F\subset\Dom(D^{2}_{T}),}{\dim F=k}} \sup_{\stackrel{s\in F,}{\|s\|=1}}
\big\langle D^{2}_{T}s, s\big\rangle.
\end{align}
Take $F=E^{j}_{T}$, then $\textup{dim}F=c_{j}+p_{j}$. It is a consequence of (\ref{5.11a}) and (\ref{5.11}) that
for every $s\in F$,
\begin{align}\label{5.16}
\big\|D_{T}s\big\|^{2}=\big\|D_{T, 3}s\big\|^{2}\leqslant b^{2}_{1}e^{-2b_{2}T}\big\|s\big\|^{2}.
\end{align}
Then (\ref{0.2}) follows immediately from (\ref{5.15}) and (\ref{5.16}).
Suppose now that $F$ is a $(c_{j}+p_{j}+1)$-dimensional subspace of $\Dom(D^{2}_{T})\cap \Omega^{j}(M)$.
Clearly $F\cap E^{\bot}_{T}\neq \{0\}$. Let $s\in F\cap E^{\bot}_{T}$ be a nonzero element. Then  (\ref{5.11}) yields
\begin{align}\label{5.17}
\big\|D_{T}s\big\|^{2}=\big\|D_{T, 2}s\big\|^{2}+\big\|D_{T, 4}s\big\|^{2}
\geqslant \big\|D_{T, 4}s\big\|^{2} \geqslant b^{2}_{3}T^{2}\big\|s\big\|^{2}.
\end{align}
Relations (\ref{5.15}) and (\ref{5.17}) imply immediately (\ref{0.1}).
\end{proof}
\subsection{Proof of Proposition \ref{t1.2} and Theorem \ref{t1.0}}\label{s7}
Denote by $F^{C_{0}}_{T}$ the finite dimensional vector space consisting of eigenspaces
of $D^{2}_{T}|_{\Dom(D^{2}_{T})}$ associated to the eigenvalues lying in $[0, C_{0})$, i.e.,
$F^{C_{0}}_{T}=\bigoplus_{j=0}^{n}F^{C_{0}}_{T, j}$.
Denote by $P^{C_{0}}_{T}$ the orthogonal projection operator from $E$
to $F^{C_{0}}_{T}$. Since $D^{2}_{T}$ preserves the degree of $\Omega^{\bullet}(M)$, the projection $P^{C_{0}}_{T}$
maps $\Omega^{j}(M)$ onto $F^{C_{0}}_{T, j}$.
Let $J_{T}$ be the linear map from $C^{j}$ to $E^{j}_{T}$ by sending
$[p]^{\ast}$ to $\rho_{p, T}$ for $p\in C^{j}(f)\cup C^{j}_{-}(f|_{\partial M})$.
Set $e_{T}: C^{j}\rightarrow F^{C_{0}}_{T, j}$ by
$e_{T}=P^{C_{0}}_{T}J_{T}$. In view of Proposition \ref{t6.1},
the following result (\cite[Th.\,6.7]{Zhang01}) still holds.
\begin{lemma}\label{t7.1}
There exists $c>0$ such that as $T\rightarrow \infty$, for any $s\in C^{j}$,
\begin{align}\label{7.2}
\big(e_{T}-J_{T})s=O(e^{-cT})\big\|s\big\|\ \ \textup{uniformly on}\ M.
\end{align}
In particular, $e_{T}$ is an isomorphism when $T$ is large enough.
\end{lemma}
\begin{proof} The proof is similar to that of \cite[Th.\,6.7]{Zhang01}.
\end{proof}
\begin{proof}[Proof of Proposition \ref{t1.2}]
From Lemma \ref{t7.1}, there exists $T_{6}>0$ such that when $T\geqslant T_{6}$,
\begin{align}
\textup{dim}\ F^{C_{0}}_{T}\geqslant \textup{dim}\ E_{T}.
\end{align}
We next carry on nearly word by word as in \cite[pp. 86-88]{Zhang01} to get
\begin{align}
\textup{dim}\ F^{C_{0}}_{T, j}=c_{j}+p_{j}.
\end{align}
Then the proof of Proposition \ref{t1.2} is complete.
\end{proof}
\begin{proof}[Proof of Theorem \ref{t1.0}]
Substituting (\ref{7.7b}) into (\ref{7.7a}), we get for $\alpha\in F^{C_{0}}_{T, j}$,
\begin{align}
P_{\infty, T}(\alpha)=\sum_{p\in C^{j}(f)\cup\, C_{-}^{j}(f|_{\partial M})}[p]^{\ast}\int\limits_{\overline{W^{u}(p)}}
e^{Tf}\alpha.
\end{align}
Note that $P_{\infty, T}$ is a chain homomorphism, i.e., $P_{\infty, T}d_{T}=\partial P_{\infty, T}$. Indeed, the proof of
\cite[Prop.\,6]{Laudenbach92} (in the boundaryless case) goes through also for
$P_{\infty, T}$ defined in (\ref{7.7a}); we just apply the Stokes formula as in \cite[Prop.\,6]{Laudenbach92}
also for the closure of unstable manifolds $\overline{W^{u}(p)}$ for $p\in C_{-}(f|_{\partial M})$.
Define $\mathcal{F}\in \textup{End}(C^{j})$ by sending $[p]^{\ast}$ to $f(p)\cdot [p]^{\ast}$
for $p\in C^{j}(f)$ and sending $[q]^{\ast}$ to $\big(f(q)+\frac{1}{2T}\ln 2\pi \big)[q]^{\ast}$
for $q\in C^{j}_{-}(f|_{\partial M})$.
Set $\mathcal{N}\in \textup{End}(C^{j})$ by taking $[p]^{\ast}$ to $j\cdot [p]^{\ast}$
for $p\in C^{j}(f)$ and taking $[q]^{\ast}$ to $(j-\frac{1}{2})\cdot [q]^{\ast}$
for $q\in C^{j}_{-}(f|_{\partial M})$.
From (\ref{2.32a}), (\ref{2.32b}) and Proposition \ref{t6.1}, we get the following analogue of
\cite[Th.\,6.9]{Zhang01}: there exists $c>0$ such that as $T\rightarrow \infty$,
\begin{align}\label{7.7}
P_{\infty, T}e_{T}=e^{T\mathcal{F}}(\frac{\pi}{T})^{\frac{\mathcal{N}}{2}-\frac{n}{4}}\big(1+O(e^{-cT})\big)
\end{align}
\noindent In particular, $P_{\infty, T}$ is an isomorphism when $T$ is large enough.
Since $P_{\infty, T}$ is a chain homomorphism,
it induces an isomorphism between the cohomology groups of the two complexes.
We finish the proof of Theorem \ref{t1.0}.
\end{proof}
\begin{rem}
Given the Morse function $f$, the Morse-Smale complex $(C^{\bullet}, \partial)$ defined as in
(\ref{2.32c}) and (\ref{2.32d}) depends on the vector field $X$. However, by
\cite[\S 2.3, Prop.]{Laudenbach10}, the homology of $(C^{\bullet}, \partial)$ is
independent of the choice of the pseudo-gradient vector field. Therefore, the Witten
complex $(F^{C_{0}}_{T, \bullet}, d_{T})$ is quasi-isomorphic to all Morse-Smale complex
constructed by Laudenbach in \cite{Laudenbach10}.
\end{rem}
\begin{rem}\label{t7.2}
For the relative boundary case, our strategy proceeds as follows.
We first find a vector field $Y$ on $M$ satisfying
the conditions (1)--(5) in Section \ref{s5} except that the sets $C^{j}_{-}(f|_{\partial M})$,
$C_{-}(f|_{\partial M})$ there should be replaced by the sets $C^{j-1}_{+}(f|_{\partial M})$ and
$C_{+}(f|_{\partial M})$, respectively, and the condition (4) should read as
\\
$(4)'$  if $p\in C_{+}^{j-1}(f|_{\partial M})$, there are coordinates
$x=(x', x_{n})\in \mathbb{R}^{n-1}\times \mathbb{R}_{+}$
on some neighborhood $U_{p}$ of $p$ such that on $U_{p}$,
\begin{align}
f(x)=f(p)-\frac{x^{2}_{1}}{2}-\ldots-\frac{x^{2}_{j-1}}{2}+\frac{x^{2}_{j}}{2}+\ldots+\frac{x^{2}_{n-1}}{2}-x_{n}
\end{align}
and
\begin{align}
Y=\sum_{i=1}^{j-1}x_{i}\frac{\partial}{\partial x_{i}}-\sum_{i=j}^{n-1}x_{i}\frac{\partial}{\partial x_{i}}
+x_{n}\frac{\partial}{\partial x_{n}}.
\end{align}
Then the corresponding Thom-Smale complex is derived by replacing
the sets $C^{j}_{-}(f|_{\partial M})$ in (\ref{2.32c}),
$C^{j+1}_{-}(f|_{\partial M})$ in (\ref{2.32d})
by the sets $C^{j-1}_{+}(f|_{\partial M})$ and $C^{j}_{+}(f|_{\partial M})$, respectively.
On the other hand, we consider different boundary conditions for $D_T^2$, that is,
a different self-adjoint extension. We start with the weak maximal extension of $\delta_T$,
\begin{equation}
\Dom(\delta_T)=\big\{w\in L^2\Omega(M),\delta_Tw\in L^2\Omega(M)\big\}
\end{equation}
where $\delta_Tw$ is calculated in the sense of distributions.
We denote by $\delta_T^*$ the Hilbert space adjoint of $\delta_T$. Hence
\begin{equation}\label{5.1}
\begin{split}
\Dom&(\delta_T^*)\cap\Omega(M)=\big\{w\in \Omega(M):w_{\textup{tan}}=0\big\}\,,\\
&\delta_T^*w=d_Tw\:\:\text{for $w\in\Dom(\delta_T^*)\cap\Omega(M)$}\,.
\end{split}
\end{equation}
The domain of the extension $D_T=\delta_T^*+\delta_T$ of deformed de
Rham operator $d_T+\delta_T$ is $\Dom(\delta_T^*)\cap\Dom(\delta_T)$.
By \eqref{5.1},
\begin{align}\label{5.7}
\textup{Dom}(D_{T})\cap\Omega(M)=\big\{w \in \Omega(M), w_{\textup{tan}}=0 \ \ \textup{on} \ \partial M\big\}.
\end{align}
%
%
%
We define the Gaffney estension of $D^{2}_{T}$ as in \eqref{1.22} by
$D^{2}_{T}=\delta_T^*\delta_T+\delta_T\delta_T^*$. Then
\begin{align}
\Dom(D^{2}_{T})\cap\Omega(M)=\Big\{ w\in\Omega(M),
\begin{split} w_{\textup{tan}}&=0 \ \
\\   \big(\delta_{T} w\big)_{\textup{tan}}&
=0  \end{split} \ \ \textup{on} \ \partial M  \Big\}.
\end{align}
The analogues of Theorems \ref{t1.3} and \ref{t1.4} are obtained simply by replacing the number $p_{j}$
in the expressions (\ref{0.1}) and (\ref{0.2}) by the number $q_{j-1}$.
Then we obtain the corresponding Witten instanton complex. Moreover, the chain morphism
between the Witten instanton complex and the Thom-Smale complex is constructed
as in (\ref{7.7b}) and (\ref{7.7a}) except that
the set $C_{-}(f|_{\partial M})$ in (\ref{7.7b}) should be replaced by the set $C_{+}(f|_{\partial M})$.
This morphism is an isomorphism for $T$ large enough. As a by-product, we
obtain the inequalities (\ref{1.1a}).
\end{rem}
\noindent
\textbf{\emph{Acknowledgements.}}
The author is indebted to Prof.\ Xiaonan Ma and Prof.\ George Marinescu for
their kind advices.

\end{document}